\documentclass[fleqn,reqno,11pt]{amsart}
\usepackage[pdftex,marginratio=1:1,width=130mm,height=190mm,top=40mm,headsep=17mm,head=7mm,footskip=11.29mm]{geometry}
\usepackage[colorlinks,linkcolor=blue,citecolor=magenta,anchorcolor=blue,bookmarksopen,pdfpagetransition={Wipe}]{hyperref}
\usepackage{graphicx}
\usepackage{amsthm}

\newtheorem{thm}{Theorem}[section]

\newtheorem{lem}[thm]{Lemma}

\theoremstyle{definition}

\numberwithin{equation}{section}

\newcommand{\dd}{\,\mathrm{ds}}
\newcommand{\B}{\bf x}
\newcommand{\bu}{\bf u}
\newcommand{\dx}{\,\mathrm{d{\bf x}}}
\begin{document}
\unitlength 1cm\begin{picture}(0,0)\put(2,1){ \scriptsize
\textbf{(CMDE)} \textsl{Computational Methods for Differential
Equations}}\end{picture}

\title[Extending a new two-grid waveform relaxation ]{Extending a new two-grid waveform relaxation on a spatial finite element discretization}%
\author[N. Habibi]{Noora Habibi$^*$}%
\address{Faculty of Applied Mathematics, Shahrood University Of Technology, P.O. Box 3619995161 Shahrood(Iran).}%
\email{habibi85nu@gmail.com}%
\thanks{$*$ corresponding}

\author[A. Mesforoush]{Ali Mesforoush}%
\address{Faculty of Applied Mathematics, Shahrood University Of Technology, P.O. Box 3619995161 Shahrood(Iran).}%
\email{ali.mesforush@shahroodut.ac.ir}%


\begin{abstract}
In this work,  a new two-grid method presented for the elliptic partial differential equations is generalized to the time-dependent linear parabolic partial differential equations. The new two-grid waveform relaxation method uses the numerical method of lines, replacing any spatial derivative by a discrete formula, obtained here by the finite element method. A convergence analysis in terms of the spectral radius of the corresponding two-grid waveform relaxation operator is also developed. Moreover, the efficiency of the presented method and its analysis are tested, applying  the two-dimensional heat equation.
\end{abstract}
\maketitle
\section{Introduction}
The design of an efficient solver which reduce the computational cost by preserving high convergence rate is one of the challenges on these decades for the numerical simulation to the partial differential equations (PDE). To this purpose, multigrid method is one of the most popular technique to accelerate the convergence  of the iterative methods to solve the system of equations.
In \cite{22, 23}, Federenko presented the first idea and algorithm of the two-grid method. After his work, many other researchers worked on this subject. Convergence analysis, gain full multigrid (FMG), nonlinear multigrid and algebraic multigrid (AMG) are achievements of these efforts (see \cite{5,10,8,9,Mc, 30,A C}). We have many signs of progress on the multigrid method until now.  One of the last work is subjected to a work presented by Moghaderi and Dehghan in \cite{moghadery}. They presented a fast and efficient two-grid method for solving the Poisson equation. Extending their idea for time-dependent equations is the aim of our work.

Let us consider as model problem the simplest time-dependent PDE, heat equation, with homogeneous Dirichlet boundary conditions
\begin{equation}\label{1.1}
\begin{split}
&D_tu({\B},t)-\Delta u({\B},t))= f({\B},t),\quad {\B} \in \Omega,\quad t>0,\\
&u({\B},t)=0,\quad \text{on } \partial \Omega,\quad t>0,\\
&u({\B},0)=g({\B}), \quad {\B}\in \Omega,
\end{split}
\end{equation}
where $\Omega \subset \mathbb{R}^2$,  is a bounded domain with boundary $\partial \Omega$. 
In order to establish the finite element approximation of our problem, let $\Omega_h$ be a triangulation of $\Omega$, satisfying the usual admissibility
assumption, i.e. the intersections of two different elements is either empty, a vertex, or a whole edge.
  Let $V_h$ be the finite element space of continuous piecewise linear functions associated with $\Omega_h$ vanishing on the boundary $\partial \Omega$. The discrete approximation $u_h\in V_h$  solves the following problem
 \begin{align*}
\left(D_tu_h,v_h\right)+a\left(u_h,v_h\right)=(f,v_h),\quad  v_h\in V_h,
\end{align*}
where
\begin{align*}
&\left(D_tu_h,v_h\right) = \int_{\Omega} (D_tu_h) v_h\, \dx,\quad  \quad (f,v_h) = \int_{\Omega} fv_h \, \dx,\\
&a\left(u_h,v_h\right) = \int_{\Omega}\nabla u_h \cdot \nabla v_h\, \dx.
\end{align*}
Let \{$\phi_1, \ldots ,\phi_{N}$\} be the nodal basis of $V_h,$ i.e., $\phi_i({\B}_j) = \delta_{ij}$, with ${\B}_j$ an interior node of the mesh ${\Omega}_h$. 
The approximation $u_h=\sum_{i=1}^{N} u_i(t)\phi_i({\B})$ is found by solving the following set of equations,
\begin{align*}
\left(D_tu_h,\phi_j\right)+a\left(u_h,\phi_j\right)=(f,\phi_j),\quad \text{for }j=1,2,\ldots,N.
\end{align*}
We rewrite these equations in terms of the mass matrix $B=\{(\phi_i,\phi_j)\}$ and the stiffness matrix $A=\{a(\phi_i,\phi_j)\}$, in a more standard form, as a system of ordinary differential equations (ODEs)
\begin{align}\label{1.2}
B_h\dot{{\bu}}_h(t)+A_h{\bu}_h(t)=F_h(t),\quad {\bu}_h(0)=g_h,\quad t>0,
\end{align}
where ${\bu}_h(t)=[u_1(t),u_2(t),\ldots,u_{N}(t)]^T\in \mathbb{R}^{N}$ and the coefficient matrices $A_h, B_h\in \mathbb{R}^{N\times N}$ and the right hand side  $F_h(t)=[(f,\phi_1),(f,\phi_2),\ldots,$ $(f,\phi_{N})]^T$ $ \in \mathbb{R}^{N}$ are considered.

We have many choices to pick a suitable method for solving the obtained ODE system \eqref{1.2}. In general, we can divide all the well known methods into two classes: time-marching approaches and time parallel techniques. The algorithm proposed here is directly applied on the ODE problem \eqref{1.2} using the waveform relaxation method (WR). The proposed waveform relaxation technique works parallel in time.

The waveform relaxation is essentially a continuous-in-time algorithm for solving ordinary differential equations. Although utilizing semi discretization on  the original PDE, it can be practical for these types of equations, too.
The WR method is based on splitting matrices $A_h$ and $B_h$ as $B_h=M_{B_h}-N_{B_h}$ and $A_h=M_{A_h}-N_{A_h}$, leading to the following iteration
\begin{align}\label{eq1}
M_{B_h}\dot{{\bu}}_h^{(\nu)}(t)+M_{A_h}{\bu}_h^{(\nu)}(t)=N_{B_h}\dot{{\bu}}_h^{(\nu-1)}(t)+N_{A_h}{\bu}_h^{(\nu-1)}(t)+F_h(t),
\end{align}
where $ {\bu}_h^{(\nu)}(0)=g_h,$ for $\nu\geq 1$ and ${\bu}_h^{(\nu)}(t)$ indicates the approximation of ${\bu}(t)$ at iteration $\nu$. It is natural to define ${\bu}_h^0(t)$ along the whole time interval equal to the initial condition, i.e. ${\bu}_h^0(t)=g_h,\,\, t>0$. 
Considering the decomposition of matrices $A_h$ and $B_h$  as $A_h=-L_{A_h}+D_{A_h}-U_{A_h}$ and $B_h=-L_{B_h}+D_{B_h}-U_{B_h}$, the the Gauss-Seidel waveform relaxation splittings in \eqref{eq1} are as follows:
\begin{align*}
\begin{matrix}
M_{A_h}=-L_{A_h}+D_{A_h}  &  N_{A_h}=U_{A_h},\\
M_{B_h}=-L_{B_h}+D_{B_h}  &  N_{B_h}=U_{B_h},
\end{matrix}
\end{align*}
where $L_{A_h}$ and $L_{B_h}$ are strictly lower triangular matrices, $D_{A_h}$ and $D_{B_h}$ are diagonal matrices, and $U_{A_h}$ and $U_{B_h}$ are strictly upper triangular matrices.
 
 In \cite{Miekkala} Miekkala and Nevanlinna showed that convergence of this method could be slow. In \cite{Lubich} Lubich and Ostermann combined waveform relaxation by the multigrid acceleration method and presented theoretical results. The mentioned works were based on spatial finite difference discretization. The theoretical convergence analysis extensions of their work to the spatial finite element discretization were the subject of \cite{JAN1, JAN2}. They studied the convergence of the standard waveform relaxation by the spectral radius of its operator which is of linear Volterra convolution type.

We use the classical \cite{ulrich} and new two-grid technique to accelerate the convergence of the Gauss-Seidel waveform relaxation method. A classical multigrid acceleration of this method was firstly studied by Lubich and Ostermann in \cite{9} and  developed in \cite{10}, independently. 
 In order to apply a classical geometric two-grid waveform relaxation procedure the coarsening applies only on the spatial domain  and  we consider a hierarchy of two grids, defined as $\Omega_{2h}\subset\Omega_{h}$.  We obtain a new iterate ${\bu}_h^{(\nu)}$ from the former waveform ${\bu}_h^{(\nu-1)}$ in three steps: Pre-smoothing, coarse grid correction and post smoothing. In Algorithm 1 we present the classical two-grid waveform relaxation algorithm (TGW) depending on the defined Gauss-Seidel  waveform relaxation method as smoother and the rest of the operators involved in the multigrid procedure. We consider  standard coarsening for constructing the coarse meshes and discretization coarse grid approximation (DCA) in coarser grids. Regarding intergrid transfer operators, the interpolation operators are the nine point stencil operator corresponding to the linear interpolation for the two dimensional problem. The restriction operators are considered as the adjoint of the prolongation operators.
\begin{table}[ht]
\small
\begin{tabular}{l}
  \hline
  {\bf Algorithm 1} The classical two-grid waveform relaxation based on the \\
  Gauss-Seidel smoother 
  $ TGW({\bu}_h^{(\nu)}(t),F_h(t),\nu_1,\nu_2)=  {\bu}_h^{(\nu+1)}(t)$ \\
  \hline
  {\bf If} we are on the coarsest grid, then solve 
   the following   equation by a direct\\ or fast solver\\
  $\hspace{1cm}B_{h_0}\dot{{\bu}}_{h_0}^{\nu+1}(t)+A_{h_0}{\bu}_{h_0}^{\nu+1}(t)=F_{h_0}(t)$.\\
  {\bf Else}\\
  (\textit{Presmoothing})  Perform $\nu_1$ steps of Gauss-Seidel waveform relaxation, \\
  $\hspace{1cm}v_h^\nu(t)=S^{\nu_1}\left({\bu}_h^{\nu}(t) \right)$.\\
  (\textit{Coarse grid correction})\\
 $\hspace{.88cm}$Compute the defect\\
   $\hspace{1cm}\bar{d}_h^\nu(t)=F_h(t)-B_h\dot{v}_h^{\nu}(t) - A_hv^{\nu}_h(t)$\\
  $\hspace{1cm}$Restrict the defect\\ 
  $\hspace{1cm}\bar{d}_{2h}^\nu(t)=r\bar{d}_h^\nu(t)$\\
  $\hspace{1cm}$Solve the following defect equation,\\
   $\hspace{1cm}B_{2h}\dot{{e}}_{2h}^\nu(t)+A_{2h}{e}_{2h}^\nu(t)=\bar{d}_{2h}^\nu(t),\,\,{e}_{2h}^\nu(0)=0$\\
   $\hspace{1cm}$Interpolate the correction\\
   $\hspace{1cm}{e}_h^\nu(t)=p{e}_{2h}^\nu(t)$\\
  $\hspace{.9cm}$Correct the current approximation with the interpolation of the correction, \\
  $\hspace{1cm}v_h^{\nu+1}(t)=v_h^{\nu}(t)+{e}_{h}^\nu(t)$. \\
   (\textit{Postsmoothing}) Perform $\nu_2$ steps of Gauss-Seidel waveform relaxation, \\
  $\hspace{1cm}{\bu}^{\nu+1}_h(t)=S^{\nu_2}\left(v_h^{\nu+1}(t) \right)$.\\
  {\bf End If}\\
  \hline
\end{tabular}\label{alg1}
\end{table} 
 
In Algorithm 1,  using the Crank-Nicolson approach for time discretization we obtain a space-time two-grid method with coarsening only in space. Thus, we have time-line Gauss-Seidel waveform relaxation, with standard full weighting restriction and bilinear interpolation in space for data transfer between the levels in the multigrid hierarchy.

Now, we present our new two-grid waveform relaxation method (NTGW). This method is based on the two-grid scheme presented by  Algorithm 1. 
 As before we consider two nested grids $\Omega_{2h}$ and $\Omega_h$, with $\Omega_{2h}\subset\Omega_h$ and we obtain new iterate ${\bu}^{(\nu)}$ from the former   one, ${\bu}^{(\nu-1)}$.

Suppose the initial guess for solving equation \eqref{eq1} be of the form:
$$\tilde{x}(t)=x^{(\nu_1)}(t)+\mathrm{ w}\bar{x}(t),$$
considering $x^{(\nu_1)}(t)=\mathcal{K}^{\nu_1}x_0(t)$  with $x_0(t)$ an initial guess and $\mathrm{w}\neq0$  a real number that we will identify it later,  and $\mathcal{K}^{\nu_1}$ an iterative matrix of an iterative method, considered here as Gauss-Seidel method, where $\nu_1$ is the number of iterations.  $\bar{x}(t)$ can be computed using two-grid waveform relaxation for solving the following initial value problem
\begin{align*}
B_h\dot{\bu}+A_h{\bu}=\frac{1}{\mathrm{w}^2}(f_h-B_h\dot{x}^{(\nu_1)}-A_hx^{(\nu_1)}).
\end{align*} 
Then we use the classical two-grid waveform relaxation on \eqref{1.2} by the initial guess $\tilde{x}$. The algorithms of the  new two-grid waveform relaxation scheme are presented in the following.

\begin{table}[ht]
\small
\begin{tabular}{l}
  \hline
  {\bf Algorithm 2} The new two-grid waveform relaxation\\
   $NTGW({\bu}_h^{(\nu-1)},f_h,w,\nu_1,\nu_2,\nu_3)\rightarrow  {\bu}_h^{(\nu)}$\\
  \hline
  1)~~Set $x_0={\bu}_h^{(\nu-1)}$, and $x^{(\nu_1)}=\mathcal{K}^{(\nu_1)}x_0$.\hspace*{5.0cm} \\
  2)~~Put $f'_h=\frac{1}{ w^2}(f_h-B_h\dot{x}^{(\nu_1)}-A_hx^{(\nu_1)})$. \\
  \hspace*{1.58cm}$=\frac{1}{ w^2}[N_{B_h}(\dot{x}^{(\nu_1)}-\dot{x}^{(\nu_1-1)})+N_{A_h}(x^{(\nu_1)}-x^{(\nu_1-1)})]$.\\
  3)~~Compute $\bar{x}$ by $ TGW(x^{(\nu_1)},f'_h,0,\nu_2)\rightarrow \bar{x}$.\\
  4)~~$\tilde{x}=x^{(\nu_1)}+{ w}\bar{x}$.\\
  5)~~Compute ${\bu}_h^{(\nu)}$ by $ TGW(\tilde{x},f_h,0,\nu_3)\rightarrow u^{(\nu)}$. \\
  \hline
\end{tabular}
\end{table}

This paper is organized as follows. In section 2, convergence analysis of a general successive approximation scheme is presented. The analysis is based on the theory Volterra integral equation together with Laplace transformation with respect to the temporal parameter to make this approach applicable to a time-dependent problem which is considered in this work. Section 2.1 is devoted to present the convergence analysis of the classical two-grid waveform relaxation method. In section 2.2, we describe the extension of this analysis matching the new two-grid waveform relaxation. We validate our theoretical results of the new two-grid waveform relaxation by numerical simulations in section 3. We compared three methods, Guass-Seidel waveform relaxation, classical two-grid waveform relaxation and the new two-grid waveform relaxation obtaining their averaged convergence factor. 

\section{convergence analysis}
The convergence analysis of the two-grid waveform relaxation method presented here is based on the theory of Volterra integral equations together with Laplace transformation with respect to the temporal parameter to reduce the time-dependent problem to a set of time-independent problems with a complex parameter. Considering a general iteration scheme $\mathcal{H}$ as a successive approximation scheme $u^{(\nu)}=\mathcal{H}u^{(v-1)}+\phi,$ we state the main theorems used to prove the convergence results where $\mathcal{H}$ is defined as
 \begin{align}
 \mathcal{H}u=Hu+\mathcal{H}_cu,
 \end{align}
with complex matrix $H$, the linear Volterra  convolution operator $\mathcal{H}_c$  and the matrix-valued kernel $h_c$ that is:
\begin{align*}
\mathcal{H}_cu(t)=h_c\star u(t)=\int_0^th_c(t-s)u(s)\,ds.
\end{align*}
We recall that convergence of the $\mathcal{H}$ will be guaranteed  if and only if the spectral radius of it be smaller than one. Equivalently, we can consider some more applicable conditions stated by the following  Lemmas. In the case of finite time-interval, we have:
\begin{lem}\cite{JAN1}\label{lemma2.1}
Suppose $h_c\in C[0,T]$ and consider $\mathcal{H}$ as an operator in $C[0,T]$. Then, $\mathcal{H}$ is a bounded operator and $\rho(\mathcal{H})=\rho(H)$.
\end{lem}
 Also, the next  Lemma indicates the spectral radius of operator $\mathcal{H}$ in the case of infinite-time intervals.
\begin{lem}\cite{JAN1}\label{lemma2.2}
Suppose $h_c\in L_1(0,\infty)$, and consider $\mathcal{H}$ as an operator in $L_p(0,\infty)$ with $1\leq p\leq \infty$. Then, $\mathcal{H}$ is a bounded operator with spectral radius 
\begin{equation}
\begin{split}
\rho(\mathcal{H})&=\sup_{Re(z)\geq 0}\rho({\bf H}(z))\\
                 &=\sup_{\xi\in \mathbb{R}}\rho({\bf H}(i\xi)),
\end{split}
\end{equation}
where ${\bf H}(z)=H+{\bf H}_c(z)$, and ${\bf H}_c(z)$ denotes the Laplace-transform of $h_c$.
\end{lem}
\subsection{Classical two-grid convergence }We can  rewrite the classical two-grid cycle as an explicit successive approximation method $u^{(\nu)}=\mathcal{M}u^{(\nu-1)}+\phi$ forcing the analytical solution \eqref{1.2} $u(t)=e^{-B^{-1}At}u_0+\int_0^te^{B^{-1}A(s-t)}B^{-1}f(s)\dd$, to the TGW algorithm (Algorithm~\ref{alg1}) \cite{JAN1}. So, the classical two-grid waveform relaxation operator $\mathcal{M}$ yields
\begin{align}\label{4.3}
\mathcal{M}u(t)=\mathcal{K}^{\nu_2}\mathcal{CK}^{\nu_1}u(t),
\end{align}
where $\mathcal{K}$ is the standard waveform relaxation operator given by
\begin{align}\label{k}
\mathcal{K}u(t)=M_B^{-1}N_Bu(t)+\mathcal{K}_cu(t)
\end{align}
and $\mathcal{K}_c$ is a linear Volterra convolution operator with kernel $k_c$, defined as
\begin{align}
\nonumber &\mathcal{K}_cu(t)=k_c\star u(t)=\int_0^t k_c(t-s)u(s)\,ds,\\
          &k_c(t)=e^{-M_B^{-1}M_At}M_B^{-1}(N_A-M_AM_B^{-1}N_B).
\end{align} 
 Also, $\mathcal{C}$  is the two-grid correction waveform operator
\begin{align}
\mathcal{C}u(t)=(I-pB_H^{-1}rB_h)u(t)+\mathcal{C}_cu(t),
\end{align}
where the operator $\mathcal{C}_c$ is of linear Volterra convolution type with matrix valued kernel  $c_c(t)=pe^{-B_H^{-1}A_Ht}B_H^{-1}(A_HB_H^{-1}rB_h-rA_h)$. So we can rewrite operator $\mathcal{M}$ in (\ref{4.3}) as 
\begin{align*}
\mathcal{M}u(t)=(M_{B_h}^{-1}N_{B_h})^{\nu_2}(I-pB_H^{-1}rB_h)(M_{B_h}^{-1}N_{B_h})^{\nu_1}u(t)+\mathcal{M}_cu(t),
\end{align*}
where the operator $\mathcal{M}_cu(t)$ is a linear combination of products of linear Volterra convolution operators $\mathcal{K}_c$ and $\mathcal{C}_c$.
 Its kernel and the Laplace transform of this kernel are denoted by $m_c(t)$ and ${\bf M}_c(z)$.
 
 Considering $e^{(\nu)}$ as the error of the $\nu$th two-grid waveform relaxation iterate, $e^{(\nu)}=\mathcal{M}e^{(\nu)}$, we can obtain the Laplace transform of it as 
 \begin{align*}
\tilde{e}^{(\nu)}(z)&=[ 
(M_{B_h}^{-1}N_{b_h})^{\nu_2}(I-pB_H^{-1}rB_h)(M_{B_h}^{-1}N_{b_h})^{\nu_1}\\
& +  {\bf M}_c(z) ]\tilde{e}^{(\nu-1)}(z) =
{\bf M}(z)\tilde{e}^{(\nu-1)}(z),
 \end{align*}
where
\begin{align}
&{\bf M}(z)={\bf K}^{\nu_2}(z)(I-p(zB_H+A_H)^{-1}r(zB_h+A_h)){\bf K}^{\nu_1}(z),\\
&{\bf K}=(zM_{B_h+M_{A_h}})^{-1}(zN_{B_h}+N_{A_h}).
\end{align}

Now, we can obtain the spectral radius of the classical two-grid waveform relaxation in finite- and infinite-time intervals using  Lemmas \ref{lemma2.1} and~~\ref{lemma2.2}, respectively. To be more precise, we have the following theorems \cite{JAN1}:
\begin{thm}\label{thrm1}
(finite-time interval) the two-grid waveform relaxation operator $\mathcal{M}$ is a bounded operator in $C[0,T]$ and
\begin{align*}
\rho(\mathcal{M})=\rho\left (M_{B_h}^{-1}N_{B_h})^{\nu_2}(I-pB_H^{-1}rB_h)(M_{B_h}^{-1}N_{B_h})^{\nu_1} \right).
\end{align*}
\end{thm}
\begin{thm}\label{thrm2}
(finite-time interval) Suppose all eigenvalues of $B_H^{-1}A_H$ and $M_{B_h}^{-1}M_{A_h}$  have positive real parts and consider $\mathcal{M}$ as an operator in $L_p(0,\infty)$ with $1\leq p \leq \infty$. Then, $\mathcal{M}$ is a bounded operator with spectral radius 
\begin{align*}
\rho(\mathcal{M}) &= \sup_{Re(z)\geq 0} \rho({\bf M}(z))\\
                  &=\sup_{\xi \in \mathbb{R}}\rho({\bf M}(i\xi)).
\end{align*}
\end{thm}
\subsection{New two-grid convergence }\label{S4}
In the rest of this section, first we present a  Lemma to state the new two-grid waveform relaxation as an explicit successive approximation method. Then, we show the convergence of this method by two Lemmas which are an extension of theorems \ref{thrm1} and \ref{thrm2}, for finite- and infinite-time intervals.
\begin{lem}\label{lemNew}
The NTGW cycle can be expressed as an explicit  successive approximation method, $u^{(\nu)}=\mathcal{N}u^{(\nu-1)}+\psi$, such that the operator $\mathcal{N}$ is
\begin{align}
\mathcal{N}u(t)=\mathcal{K}^{\nu_3}\mathcal{CR}\mathcal{K}^{\nu_1}u(t),
\end{align}
with $\mathcal{K}$ the standard waveform relaxation operator \eqref{k} and $\mathcal{CR}$ the new two-grid correction waveform operator
\begin{align}
\nonumber\mathcal{CR}u(t)&=\left(I-pB_H^{-1}rB_h\right)\left(I+\mathrm{w}\mathcal{K}^{\nu_2}\left(I-\left(1+\frac{1}{\mathrm{w}^2}\right)pB_H^{-1}rB_h\right)\right)u(t)\\
&\quad+\mathcal{CR}_cu(t),
\end{align}
where $\mathcal{CR}_c$ is of linear Volterra convolution  type operator.
\end{lem}
\begin{proof}
As we explained before, applying the classical two-grid scheme on the initial value problem \eqref{1.2}  with initial guess $\tilde{x}=x^{(\nu_1)}+\mathrm{w}\bar{x}$\, leads to the NTGW. So we will do the following steps:
\begin{itemize}
\item[1)] First, we present the explicit successive approximation method to compute $\bar{x}$,
\item[2)]Then, in a similar way we obtain the explicit successive approximation method to compute $u_h^{(\nu)}$,
\item[3)]Finally setting $\bar{x}$ in $\tilde{x}$ and considering  $\tilde{x}$ as the initial guess of the classical two-grid waveform relaxation, step 5 of Algorithm 2, we can compute the whole explicit successive approximation of the new two-grid waveform relaxation method to obtain new iterate $u^{(\nu)}$.
\end{itemize}

1)~~Consider the coarse-grid correction steps to compute $\bar{x}$, as follows:
\begin{itemize}
\item[i-]$d_h(t)=B_h\dot{x}^{(\nu_1)}(t)+A_hx^{(\nu_1)}(t)-f'_h(t),$
\item[ii-]$d_h(t)=rd_H(t),$
\item[iii-]$B_H\dot{v}(t)_H+A_Hv_H(t)=d_H(t),\quad v_H(0)=0,$
\item[iv-]$\bar{x}=x^{(\nu_1)}-pv_H.$
\end{itemize}
Substituting  the defect vector and the analytical solution $v_H$, that is
\[v_H(t)=e^{-B^{-1}At}v_{H}(0)+\int_0^te^{B^{-1}A(s-t)}B^{-1}d_H(s)\dd,\] 
into the step iv, we obtain
 \begin{align}
\nonumber\bar{x}&=x^{\nu_1}-p( \int_0^te^{B_H^{-1}A_H(s-t)}B_H^{-1}r [  (1+\frac{1}{\mathrm{w}^2})B_h\dot{x^{\nu_1}} \\
\nonumber& +(1+\frac{1}{\mathrm{w}^2})A_hx^{\nu_1}-\frac{1}{\mathrm{w}^2}f_h]\dd )\\
\nonumber&=x^{(\nu_1)}(t)+\frac{1}{\mathrm{w}^2}p\int_0^te^{B_H^{-1}A_H(s-t)}B_H^{-1}rf_h\dd\\
\nonumber &-(1+\frac{1}{\mathrm{w}^2})p\int_0^te^{B_H^{-1}A_H(s-t)}B_H^{-1}rA_hx^{(\nu_1)}(s)\dd\\
 \label{I}&-(1+\frac{1}{\mathrm{w}^2})p\int_0^te^{B_H^{-1}A_H(s-t)}B_H^{-1}rB_h\dot{x}^{(\nu_1)}(s)\dd,
\end{align}
where $p$ and $r$ are prolongation and restriction operators, respectively. In equation \eqref{I} we can omit the time derivative by replacing this equation by the following statement
\begin{align}
\nonumber &-(1+\frac{1}{\mathrm{w}^2})p\int_0^te^{B_H^{-1}A_H(s-t)}B_H^{-1}rB_h\dot{x}^{(\nu_1)}(s)\dd\\
\nonumber&=\int_0^t e^{B_H^{-1}A_H(s-t)}B_H^{-1}rB_h\dot{x}^{(\nu_1)}(s)\dd\pm e^{B_H^{-1}A_H(s-t)}B_H^{-1}A_HB_H^{-1}rB_hx^{(\nu_1)}\\
\nonumber&=\int_0^t\frac{d}{ds}\left( e^{B_H^{-1}A_H(s-t)}B_H^{-1}rB_hx^{(\nu_1)}(s)\right)\dd\\
\nonumber&-\int_0^te^{B_H^{-1}A_H(s-t)}B_H^{-1}A_HB_H^{-1}rB_hx^{(\nu_1)}(s)\dd\\
\label{I'}&=B_H^{-1}rB_hx^{(\nu_1)}(t)-e^{-B_H^{-1}A_Ht}B_H^{-1}rB_hx^{(\nu_1)}(0)\\
\nonumber&-\int_0^te^{B_H^{-1}A_H(s-t)}B_H^{-1}A_HB_H^{-1}rB_hx^{(\nu_1)}(s)\dd.
\end{align}
Replacing \eqref{I'} by \eqref{I} we obtain the explicit successive approximation of the $\bar{x}$, as follows,
\begin{align*}
\bar{x}=\mathcal{M}_1x^{(\nu_1)}(t)+\mathcal{K}^{\nu_2}\phi_1(t),
\end{align*}
where
\begin{align*}
&\mathcal{M}_1x^{(\nu_1)}(t)=\mathcal{K}^{\nu_2}\mathcal{C}_1x^{(\nu_1)}(t)\\
&\mathcal{C}_1x^{(\nu_1)}(t)=(I-(1+\frac{1}{\mathrm{w}^2})pB_H^{-1}rB_h)x^{(\nu_1)}(t)+\mathcal{C}_{1c}x^{(\nu_1)}(t),\\
&\mathcal{C}_{1c}x^{(\nu_1)}(t)=C_{1c}\star x^{(\nu_1)}(t)=\int_0^tC_{1c}(t-s)x^{(\nu_1)}(s)\dd,\\
&C_{1c}(t)=(1+\frac{1}{\mathrm{w}^2})pe^{-B_H^{-1}A_Ht}B_H^{-1}(A_HB_H^{-1}rB_h-rA_h),
\end{align*}
and
\begin{align*}
&\phi_1(t)=\\
      &(1+\frac{1}{\mathrm{w}^2})pe^{-B_H^{-1}A_Ht}B_H^{-1}rB_hx^{(\nu_1)}(0)+\frac{1}{\mathrm{w}^2}p\int_0^te^{B_H^{-1}A_H(s-t)}B_H^{-1}rf_h\,ds.
\end{align*}

By the same way, we can compute the explicit successive approximation of $u^{(\nu)}$ as follows,
\begin{align}\label{final}
u^{(\nu)}=\mathcal{M}_2\tilde{x}(t)+\mathcal{K}^{\nu_3}\phi_2(t),
\end{align}
such that
\begin{align*}
&\mathcal{M}_2\tilde{x}(t)=\mathcal{K}^{\nu_3}\mathcal{C}_2\tilde{x}(t),\\
&\mathcal{C}_2\tilde{x}(t)=(I-pB_H^{-1}rB_h)\tilde{x}(t)+\mathcal{C}_{2c}\tilde{x}(t),\\
&\mathcal{C}_{2c}\tilde{x}(t)=C_{2c}\star\tilde{x}(t)=\int_0^tC_{2c}(t-s)\tilde{x}(s)\dd,\\
&C_{2c}=pe^{-B_H^{-1}A_Ht}B_H^{-1}(A_HB_H^{-1}rB_h-rA_h),
\end{align*}
and
\begin{align*}
\phi_2(t)=pe^{-B_H^{-1}A_Ht}B_H^{-1}rB_h\tilde{x}(0)+p\int_0^te^{B_H^{-1}A_H(s-t)}B_H^{-1}rf(s)\,ds.
\end{align*}
Setting $\tilde{x}=x^{\nu_1}+\mathrm{w}\bar{x}$ and $x^{(\nu_1)}=\mathcal{K}_1u^{(\nu-1)}$ into the equation (\ref{final}) we obtain
\begin{align*}
u^{(\nu)}&=\mathcal{N}u^{(\nu-1)}+\psi,\\
\mathcal{N}u(t)&=\mathcal{K}^{\nu_3}\mathcal{CR}\mathcal{K}^{\nu_1}u(t),\\
\mathcal{CR}u(t)&=(I-pB_H^{-1}rB_h)(I+\mathrm{w}\mathcal{K}^{\nu_2}(I-(1+\frac{1}{\mathrm{w}^2})pB_H^{-1}rB_h))u(t)\\&+\mathcal{CR}_cu(t),\\
\mathcal{CR}_cu(t)&=C_{2c}\star(I+\mathrm{w}\mathcal{K}^{\nu_2}(I-(1+\frac{1}{\mathrm{w}^2})pB_H^{-1}rB_h))u(t)\\
&+\mathrm{w}(I-pB_H^{-1}rB_h)\mathcal{K}^{\nu_2}(C_{1c}\star u(t))\\
&+\mathrm{w}C_{2c}\star ( \mathcal{K}^{\nu_2}C_{1c}\star u(t)).
\end{align*} 
where $\psi=\mathrm{w}\mathcal{K}^{\nu_3}\mathcal{C}_2\mathcal{K}^{\nu_2}\phi_1+\mathcal{K}^{\nu_3}\phi_2$. So
\begin{align*}
\mathcal{N}u(t)&=(M_{B_h}^{-1}N_{B_h})^{\nu_3}(I-pB_H^{-1}rB_h)\times\\
&(I+\mathrm{w}(M_{B_h}^{-1}N_{B_h})^{\nu_2}(I-(1+\frac{1}{\mathrm{w}^2})pB_H^{-1}rB_h))(M_{B_h}^{-1}N_{B_h})^{\nu_1}u(t)\\
&+\mathcal{N}_cu(t).
\end{align*}
Operator $\mathcal{N}_c$ is a linear combination of product of linear Volterra convolution operator $\mathcal{K}_c$ and $\mathcal{CR}_c$. Thus it is itself of linear Volterra convolution type.
\end{proof}

Let $e^{(\nu)}$ be the error of the $\nu$-th  iteration of the new two-grid waveform relaxation method, i.e., $e^{(\nu)}=u^{(\nu)}-u$. Using lamma \ref{lemNew}, we can conclude $e^{(\nu)}=\mathcal{N}e^{(\nu-1)}$. Therefore the Laplace transform of this equation is as follows
\begin{align*}
\tilde{e}^{(\nu)}&={\bf N}\tilde{e}^{(\nu-1)},\\
{\bf N}&={\bf K}^{\nu_3}{\bf R}{\bf K}^{\nu_1},\\
{\bf K}&=(zM_{B_h}+M_{A_h})^{-1}(zN_{B_h}+N_{A_h}),\\
{\bf R}&=(I-pB_H^{-1}rB_h)(I+\mathrm{w}{\bf K}^{\nu_2}(I-(1+\frac{1}{\mathrm{w}^2})B_H^{-1}rB_h))\tilde{e}^{(\nu-1)}\\
       &+{\bf L}(I+\mathrm{w}{\bf K}^{\nu_2}(I-(1+\frac{1}{\mathrm{w}^2})B_H^{-1}rB_h))\tilde{e}^{(\nu-1)}\\
       &+\mathrm{w}(I-pB_H^{-1}rB_h){\bf K}^{\nu_2}[(1+\frac{1}{\mathrm{w}^2}){\bf L}]\tilde{e}^{(\nu-1)}+\mathrm{w}{\bf L}[(1+\frac{1}{\mathrm{w}^2}){\bf L}]\tilde{e}^{(\nu-1)}\\
       &+{\bf L}\tilde{e}^{(\nu-1)},\\
{\bf L}&=p(A_H+zB_H)^{-1}(A_HB_H^{-1}rB_h-rA_h).\\
\end{align*}

Now, we can conclude the convergence of the new two-grid waveform relaxation method in the finite time interval by using  Lemma \ref{lemma2.1}.
\begin{lem}\label{1}
The new two-grid waveform relaxation operator $\mathcal{N}$ is a bounded operator in $C[0,T]$ and
\begin{equation}\label{rho}
\begin{split}
&\rho (\mathcal{N})=\rho((M_{B_h}^{-1}N_{B_h})^{\nu_3}(I-pB_H^{-1}rB_h)\times\\
&(I+\mathrm{w}(M_{B_h}^{-1}N_{B_h})^{\nu_2}(I-(1+\frac{1}{\mathrm{w}^2})pB_H^{-1}rB_h))(M_{B_h}^{-1}N_{B_h})^{\nu_1}).
\end{split}
\end{equation} 
\end{lem}
\begin{proof}
Both $\mathcal{K}_c$ and $\mathcal{CR}_c$ are continuous operators on $[0,T]$, So, we can conclude the kernel of $\mathcal{N}_c$ also will be in the interval $[0,T]$. Thus Lemma \ref{lemma2.1} guarantees the presented expression for the spectral radius of the operator $\mathcal{N}$.  
\end{proof}

\begin{lem}\label{2} Suppose all eigenvalues of $B_H^{-1}A_H$ and $M_{B_h}^{-1}M_{A_h}$ have positive real parts and
 $\mathcal{N}$ be an operator in $L_p(0,\infty)$ with $1\leq p< \infty $, then $\mathcal{N}$ is a bounded operator with spectral radius 
\begin{equation}
\begin{split}
\rho(\mathcal{N})&=\sup_{Re(z)\geq 0}\rho({\bf N}(z))\\
&=\sup_{\xi\in \mathbb{R}}\rho({\bf N}(i\xi)).
\end{split}
\end{equation}
\end{lem}
\begin{proof}The positivity of $B_H^{-1}A_H$ and $M_{B_h}^{-1}M_{A_h}$ can be concluded from the boundedness of 
the waveform relaxation operator $\mathcal{K}$
 and from the boundedness of  the analytical solution of the equation $\eqref{1.2}$ ( on  $\Omega_H$).
 On the other hand, we can easily obtain the equality
$${\bf N}_c(z)={\bf N}-\lim_{z\rightarrow \infty}{\bf N}(z).$$
We can conclude from positivity of real parts  of all eigenvalues of $B_H^{-1}A_H$ and $M_{B_h}^{-1}M_{A_h}$, the entries of ${\bf N}_c(z)$ are rational functions of $z$  vanishing at infinity,  all of whose poles have negative real part. Therefore, this statement together with using an inverse Laplace transformation argument implies that kernel of $\mathcal{N}_c$ is embedded in the  $ L_1(0,\infty)$. Finally, the proof will be complete  following Lemma \ref{lemma2.2}.
\end{proof}
\section{Numerical results}
In this section, we compare the numerical solution of the three methods: The Gauss-Seidel waveform relaxation, the classical two-grid waveform relaxation and the new two-grid waveform relaxation in term of convergence factors using two different examples. We consider the following two dimensional heat equation.
\begin{align*}
\frac{\partial u}{\partial t}-\Delta_2u=f,\quad (x,y)\in [0,1]\times [0,1],\quad t\in [0,1].
\end{align*}
In both cases, the classical two-grid waveform relaxation and the new two-grid waveform relaxation, we consider a linear finite element discretization on a uniform triangular spatial mesh as mentioned before leads to the following ODE system,
\begin{align*}
B_h\dot{u}_h(t)+A_hu_h(t)=F_h(t).
\end{align*}
 In the case of dealing with a structured grid, however, it suffices to represent the discrete operators utilizing stencils. The corresponding stencils obtained for the mass and stiffness matrices in two dimensions are as follows:
\[
B_h=\frac{1}{12}\left[\begin{matrix}
&1&1\\
1&6&1\\
1&1&
\end{matrix}\right], \qquad
A_h= \frac{1}{h^2}\left[\begin{matrix}
&-1&\\
-1&4&-1\\
&-1&
\end{matrix}\right].
\]
Regarding the intergrid transfer operators, the stencil of the restriction operator, $H_{h}^{2h}$, is given as follows,
\[
I_{h}^{2h}=\frac{1}{8}\left[\begin{matrix}
&1&1\\
1&2&1\\
1&1&
\end{matrix}\right].
\]
The prolongation operator, $I_{2h}^h$, is obtained according to the relation $I_{h}^{2h}=\frac{1}{4}I_{2h}^h$, \cite{ulrich}. 

The convergence factors, presented in Tables  \ref{T1} and \ref{T2}, are computed using the following division
\[
\rho^{(\nu)}=\frac{||e_h^{(\nu)}||}{||e_h^{(\nu-1)}||}.
\]

To compute the numerical solution of the NTGW, we use the optimum value of parameter $\omega$ obtained from the following algorithm \cite{moghadery}.
\begin{table}[!h]
		\begin{tabular}{l}
			\hline
			{\bf Agorithm 3} The optimum value of the parameter  $\omega$.\\
			\hline
			1) $q=(aa:hh:bb);\hspace{.5cm}(where\,\, aa=\frac{a-1}{a^2}, bb=\frac{1-a}{a^2}\,\, and \,\, 0<hh<1)$\\
			2) for $k=1:length(q)$\\
			$\qquad \omega=q(k);$\\
			$\qquad u=$ Applying  Algorithm\! 2;\\
			$\qquad error=u_{exact}-u;$\\
			$\qquad max(k)=$ Compute infinity norm of the error;\\
			$end$\\
			3) plot $(q,max).$\\
			\hline
		\end{tabular}
	\end{table}
\begin{figure}[!ht]
\caption{The optimum value of the parameter $\omega$ on mesh $49\times 49$ with $\nu_1=\nu_2=\nu_3=1$.	}\label{sh7}
\centerline{\includegraphics[width=10cm]{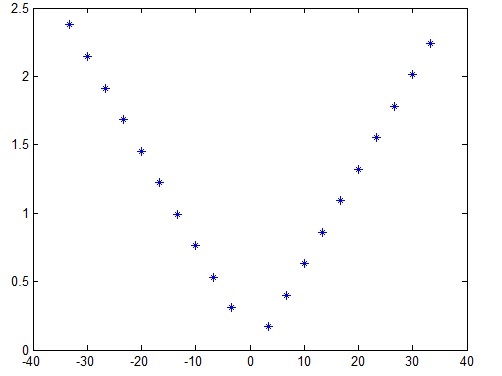}}
\end{figure}

	We illustrate the resulting values of the parameter $\omega$ in Figure \ref{sh7}  by considering spatial mesh $49\times 49$  with one iteration of smoothing step in NTGW method. Also, a
comparable results related to the different spatial discretization are presented
in Table \ref{opt_w}
\begin{table}[!ht]
\caption{\footnotesize {Optimum value of the parameter $\mathrm{w}$ for several number of $h$ and by new time-dependent two-grid method. }}\label{t2}
	\begin{tabular}{cccccc}
	\hline \hline
h&        $\frac{1}{8}$ & $\frac{1}{16}$& $\frac{1}{32}$&$\frac{1}{64}$&$\frac{1}{128}$\\
\hline
optimum value of $\mathrm{w}$&  0.8091& 0.6179& 0.1887&0.1107& 0.0847\\

\hline
\end{tabular}\label{opt_w}
\end{table}

For the first example, we consider the Dirichlet boundary and initial conditions
 in such a way that they admit in the following exact solution
 \[
 u(x,y,t)= t^2 \sin(\frac{\pi x}{2})\sin(\frac{\pi y}{2})	
 \]
We assume the finite element discretization by linear basis functions, on a fixed spatial finite element mesh size $h=2^{-7}$ together with the Crank-Nicolson time-discretization which its mesh size is varying from $0.04$ to $0.001$. We notice that the smaller time steps are more applicable to investigate the spectral radius of the continuous waveform relaxation method \cite{parallel}.
 As mentioned before, we consider the standard coarsening, full weighting restriction and linear interpolation as inter-grid transfer operators. Also, to compute coarse grid operators $A_H$ and $B_H$, we pursue the DCA discretization.
In Table \ref{T1},  we show the comparative convergence factors of the Gauss-Seidel waveform relaxation, the classical two-grid waveform relaxation and the new two-grid waveform relaxation together with the necessary iterations to reduce the initial residual by a factor $10^{-10}$ and the corresponding CPU time and also the norms of the errors, which are constant for the three schemes where the method is convergent. As you can see the behavior of the proposed new two-grid waveform relaxation method  is satisfactory. 
\begin{table}[ht]
\caption{ Results of example 1. The comparative convergence factors of the Gauss-Seidel waveform relaxation (GS), the classical two-grid waveform relaxation (TG) and the new two-grid waveform relaxation (NTG) where $h=2^{-7}$ and $\tau$ is varying from $0.04$ to $0.001$. 
}\label{T1}
\scalebox{0.76}{%
	\begin{tabular}{|c|c|cccccc|}
	\hline 
	$\tau$ & method & iter & $||.||_{\infty}$ & $||.||_{2}$ & $\rho$ & CPU & res\\ 
\hline   \hline
     &GS& $5000$ & $1.3017\times 10^{-5}$ & $1.0126\times 10^{-4}$ & $0.9991$ & $891$ s & $0.0067$\\
$0.04$& TG& $14$ & $4.2959\times 10^{-5}$ & $3.4724\times 10^{-4}$ & $0.2617$ & $10$ s  & $9.9599\times 10^{-11}$\\
     & NTG & $9$ & $4.2959\times 10^{-5}$ & $3.4724\times 10^{-4}$ & $0.1247$ & $9$ s   & $7.8106\times 10^{-11}$\\ \hline
    & GS &$5000$ & $1.3272\times 10^{-5}$ & $1.0845\times 10^{-4}$ & $0.9959$ & $876$ s & $4.3890\times 10^{-4} $\\
$0.02$& TG& $13$ & $1.4530\times 10^{-5}$ & $2.0134\times 10^{-4}$ & $0.2354$ & $8$ s  & $4.7462\times 10^{-11}$\\
      &NTG& $8$  & $1.4530\times 10^{-5}$ & $2.0134\times 10^{-4}$ & $0.1030$ & $8$ s   & $7.7974\times 10^{-11}$\\ \hline
     & GS & $3625$& $3.8156\times 10^{-6}$ & $3.1738\times 10^{-5}$ & $0.9856$ & $716$ s & $9.9589\times 10^{-11} $\\ 
$0.01$& TG& $12$  & $3.8156\times 10^{-6}$ & $3.1738\times 10^{-5}$ & $0.2368$ & $10$ s  & $2.5053\times 10^{-11}$\\
     &NTG& $7$    & $3.8156\times 10^{-6}$ & $3.1738\times 10^{-5}$ & $0.0955$ & $10$ s  & $9.4050\times 10^{-11}$\\ \hline
     & GS & $1809$& $1.0339\times 10^{-6}$ & $8.7701\times 10^{-6}$ & $0.9724$ & $429$ s & $9.7655\times 10^{-11} $\\ 
$0.005$& TG& $10$ & $1.0339\times 10^{-6}$ & $8.7701\times 10^{-6}$ & $0.2569$ & $6$ s  & $5.1234\times 10^{-11}$\\          
       &NTG&  $7$ & $1.0339\times 10^{-6}$ & $8.7701\times 10^{-6}$ & $0.0952$ & $7$ s  & $1.1617\times 10^{-11}$\\ \hline
      & GS & $895$& $2.6999\times 10^{-7}$ & $2.3256\times 10^{-6}$ & $0.9478$ & $165$ s & $9.9510\times 10^{-11} $\\ 
$0.0025$& TG& $9$ & $2.6999\times 10^{-7}$ & $2.3256\times 10^{-6}$ & $0.2188$ & $5$ s  & $2.4466\times 10^{-11}$\\          
      &NTG&  $6$ & $2.6999\times 10^{-7}$ & $2.3256\times 10^{-6}$ & $0.0946$ & $7$ s  & $1.4975\times 10^{-11}$\\ \hline   
     & GS & $353$& $4.4362\times 10^{-8}$ & $3.8769\times 10^{-7}$ & $0.8836$ & $67$ s & $9.3922\times 10^{-11} $\\ 
$0.001$& TG& $7$ & $4.4362\times 10^{-8}$ & $3.8769\times 10^{-7}$ & $0.2347$ & $4$ s  & $2.8991\times 10^{-11}$\\          
      &NTG&  $6$ & $4.4362\times 10^{-8}$ & $3.8769\times 10^{-7}$ & $0.0610$ & $4$ s  & $8.9077\times 10^{-11}$\\ \hline                
\end{tabular}
}
\end{table}

For the second example, the homogeneous Dirichlet boundary conditions and the initial condition are chosen such that the analytical solution is given by
 $$u(x,y,t)=1+\sin(\pi x/2)\sin(\pi y/2)\exp (-\pi^2t/2).$$
Again we consider a fixed spatial finite element mesh size $h=2^{-7}$ together with varying time-steps from $0.04$ to $0.001$.


In Table \ref{T2}, we have reported the comparable results of the numerical  averaged convergence factors applying three methods: Gauss-Seidel approximation, classical two-grid and new two-grid scheme together with the necessary iterations to reduce the initial residual by a factor $10^{-10}$ and the corresponding CPU time and also the norms of the errors, which are constant for the three schemes where the method is convergent. As we can see, the results concerning the new two-grid method is satisfactory. 

\begin{table}[ht]
\caption{\footnotesize {Results of example 2. The comparative convergence factors of the Gauss-Seidel waveform relaxation (GS), the classical two-grid waveform relaxation (TG) and the new two-grid waveform relaxation (NTG) where $h=2^{-7}$ and $\tau$ is varying from $0.04$ to $0.001$. 
}}\label{T2}
\scalebox{0.76}{%
	\begin{tabular}{|c|c|cccccc|}
	\hline 
	$\tau$ & method & iter & $||.||_{\infty}$ & $||.||_{2}$ & $\rho$ & CPU & res\\ 
\hline   \hline
     &GS& $5000$ & $0.0230$ & $0.1840$ & $0.9993$ & $899$ s & $0.0534$\\
$0.04$& TG& $19$ & $5.2134\times 10^{-5}$ & $0.0042$ & $0.3436$ & $10$ s  & $8.5999\times 10^{-11}$\\
     & NTG & $12$ & $5.2134\times 10^{-5}$ & $0.0042$ & $0.1399$ & $9$ s   & $9.1255\times 10^{-11}$\\ \hline
    & GS &$5000$ & $3.2446\times 10^{-4}$ & $0.0026$ & $0.9975$ & $920$ s & $0.0320$\\
$0.02$& TG& $19$ & $3.2446\times 10^{-4}$ & $0.0026$ & $0.3009$ & $9$ s  & $4.3306\times 10^{-11}$\\
      &NTG& $12$  & $3.2446\times 10^{-4}$ & $0.0026$ & $0.1379$ & $9$ s   & $4.6382\times 10^{-11}$\\ \hline
     & GS & $4237$& $1.2487\times 10^{-4}$ & $0.0010$ & $0.9844$ & $769$ s & $9.9380\times 10^{-11} $\\ 
$0.01$& TG& $18$  & $1.2487\times 10^{-4}$ & $0.0010$ & $0.2148$ & $8$ s  & $7.5243\times 10^{-11}$\\
     &NTG& $12$    & $1.2487\times 10^{-4}$ & $0.0010$ & $0.1379$ & $9$ s  & $2.3113\times 10^{-11}$\\ \hline
     & GS & $2168$& $5.1473\times 10^{-5}$ & $4.3197\times 10^{-4}$ & $0.9689$ & $433$ s & $9.7812\times 10^{-11} $\\ 
$0.005$& TG& $18$ & $5.1473\times 10^{-5}$ & $4.3197\times 10^{-4}$ & $0.2487$ & $8$ s  & $4.4703\times 10^{-11}$\\          
       &NTG&  $11$ & $5.1473\times 10^{-5}$ & $4.3197\times 10^{-4}$ & $0.1238$ & $8$ s  & $8.9182\times 10^{-11}$\\ \hline
      & GS & $1103$& $2.5114\times 10^{-5}$ & $2.2147\times 10^{-4}$ & $0.9397$ & $208$ s & $9.7242\times 10^{-11} $\\ 
$0.0025$& TG& $17$ & $2.5114\times 10^{-5}$ & $2.2147\times 10^{-4}$ & $0.2798$ & $8$ s  & $9.6319\times 10^{-11}$\\          
      &NTG&  $11$ & $2.5114\times 10^{-5}$ & $2.2147\times 10^{-4}$ & $0.1323$ & $8$ s  & $5.5868\times 10^{-11}$\\ \hline 
     & GS & $453$& $1.0493\times 10^{-5}$ & $9.6229\times 10^{-5}$ & $0.8611$ & $85$ s & $9.2321\times 10^{-11} $\\ 
$0.001$& TG& $17$ & $1.0493\times 10^{-5}$ & $9.6229\times 10^{-5}$ & $0.3095$ & $8$ s  & $6.1004\times 10^{-11}$\\          
      &NTG&  $11$ & $1.0493\times 10^{-5}$ & $9.6229\times 10^{-5}$ & $0.1690$ & $8$ s  & $5.7011\times 10^{-11}$\\ \hline                
\end{tabular}
}
\end{table}

\section{Conclusions}
This work is devoted to extending the new two-grid method introduced in \cite{moghadery} to the case when the problem is time-dependent, using the known heat equation.  They presented a new multigrid for Poisson equation based on finite   difference technique. This work is more general than their work in the following aspects:
\begin{itemize}
\item NTGW is applicable for time-dependent equations.
\item Using the finite element discretization leads to more general ODE system against the finite difference method. 
\end{itemize}
 A method of line approach together with a finite element spatial discretization is considered. The resulting ODE system is solved by means of the new two-grid waveform relaxation. Our analysis of the convergence factor is based on the spectral radius of this new two-grid waveform relaxation operator.
The convergence property of the NTGW is analyzed by Lemmas \ref{1} and~~ \ref{2} based on spectral radius of this method. As mentioned in Tables \ref{T1}, \ref{T2} among popular methods such as Gauss-Seidel and classical two-grid, the NTGW has better averaged convergence factor. 

In this work, we studied the continuous-time case of the NTGW. So, studding on the discontinuous-time case can be considered as a future work. Also, using the NTGW as ODE solver for the other time-dependent equation is applicable.

\end{document}